\documentclass[graybox]{SNmult}

\usepackage{type1cm}
\usepackage{makeidx}
\usepackage{graphicx}
\usepackage{multicol}
\usepackage[bottom]{footmisc}
\usepackage{newtxtext}        
\usepackage[varvw]{newtxmath}

\makeindex

\begin{document}

\title*{Efficient computation of eddy-currents for nonlinear magnetic field problems}

\author{Herbert Egger, 
Nepomuk Krenn, and
Andreas Schafelner}

\institute{Herbert Egger \at Institute for Numerical Mathematics, Johannes Kepler University, Altenbergerstr.~69, 4040 Linz, Austria \email{herbert.egger@jku.at}
\and 
Nepomuk Krenn \at Johann Radon Institute for Computational and Applied Mathematics, Austrian Academy of Sciences, Altenbergerstr.~69, 4040~Linz, Austria, \email{nepomuk.krenn@ricam.oeaw.ac.at} 
\and 
Andreas Schafelner \at Institute for Numerical Mathematics, Johannes Kepler University, Altenbergerstr.~69, 4040 Linz, Austria \email{andreas.schafelner@jku.at}}
%


\maketitle


\abstract{Estimation of eddy-current losses in conducting non-laminated components of electrical devices requires expensive three-dimensional simulations. Various approximations are therefore used in practice to reduce the computational cost in the early design phase. We review some approaches and discuss their modelling assumptions and resulting approximations. In particular, we identify eddy-current reaction fields as a significant contribution that should be accounted for globally. These reaction fields can be approximately reconstructed from two-dimensional magnetostatic simulations by solving a single linearized time-periodic problem. We further discuss different strategies for solving this post-processing problem. Numerical results demonstrate improved loss prediction compared to standard post-processing at moderate additional cost. 
}

\section{Introduction}
\label{egger:sec:1}

The analysis of eddy-currents by finite element methods is well established~\cite{egger:Kameari1990}. Related losses in windings and other electrical conductors are frequently represented by equivalent circuit models, which can be coupled efficiently to field formulations~\cite{egger:Tsukerman1992,egger:Demenko2010}. 
Losses in laminated iron cores can be treated by appropriate lamination models~\cite{egger:Gyselinck2006}.

A particular challenge arises in the prediction of eddy-current losses in conducting non-laminated components, such as machine housings, shielding elements, or permanent magnets in electrical motors. For these parts, the current paths are inherently three-dimensional and accurate loss prediction formally requires the solution of three-dimensional nonlinear time-dependent eddy-current problems~\cite{egger:Biro2006,egger:Takahashi2012}. To avoid the cost of such computations, 
various methods have been proposed to approximate these effects while enabling efficient simulation.
A common strategy is to compute a sequence of two-dimensional nonlinear magnetostatic or quasi-static field solutions and to estimate the associated eddy-current losses by suitable post-processing~\cite{egger:Deak2008,egger:Belahcen2010}. More sophisticated approaches combine two-dimensional field simulations with auxiliary three-dimensional eddy-current computations~\cite{egger:Yamazaki2009,egger:Okitsu2012}. 
A related approach that avoids three-dimensional computations was presented in~\cite{egger:Steentjes2015}. 
While computationally efficient, the validity and simplifying assumptions of these models are not completely clear. 
We therefore focus on eddy-current computations in two-dimensional models in this paper.

\medskip 
\textbf{Outline and main results.}
We start by revisiting some commonly used approaches for the computation of eddy-current losses in conducting non-laminated components. We recall the two-dimensional time-dependent formulation of \cite{egger:Tsukerman1992,egger:Belahcen2010} and interpret it as a consistent reduction of a class of three-dimensional eddy-current problems under specific symmetry assumptions on the problem data. This viewpoint highlights the role of the global reaction field, which is not fully captured in standard magnetostatic approximations. Based on this interpretation, we derive a linearized eddy-current problem that estimates the missing reaction field from a sequence of nonlinear magnetostatic solutions. 
This correction may be interpreted as one Newton update step for the underlying system.
We thus replace the solution of the original nonlinear time-periodic eddy-current problem in \cite{egger:Tsukerman1992,egger:Belahcen2010} by the solution of 
\begin{itemize} \itemindent1em
\item[(i)] \; a sequence of nonlinear magnetostatic problems;
\item[(ii)] \; one linearized time-periodic eddy-current problem.
\end{itemize}
The system arising in step (ii) can be solved efficiently using parallel-in-time techniques and provides a consistent way to incorporate eddy-current feedback into standard post-processing workflows.
The general approach (i)--(ii) is applicable in two and three space dimensions, but we here focus on the two-dimensional setting usually employed in electric machine design.
In numerical experiments for a permanent-magnet synchronous machine, we illustrate that direct estimation of eddy-current losses from the magnetostatic field solution in step (i) may lead to a significant overestimation of total losses (more than 100\%). 
Correcting for the reaction fields via step (ii) leads to loss predictions with excellent agreement (errors below 1\%) with the nonlinear two-dimensional time-dependent reference model \cite{egger:Tsukerman1992,egger:Belahcen2010}. While providing substantial improvement in accuracy, the additional computational cost of step (ii) is very moderate (about 30-60\% compared to the magnetostatic solutions in step (i)). 
The proposed strategy thus yields similar accuracy to the nonlinear time-periodic reference model, but at reduced computational cost (about 30-50\% for our test case). 
Before concluding, we also briefly comment on three-dimensional post-processing approaches and highlight some of their current limitations which provide research topics for future work. 
%

\section{Eddy-current models}
\label{egger:sec:2}

We start with recalling the eddy-current problems in two and three space dimensions, discuss their relations, and present some approximations considered in literature. This section introduces the general problem setting which is used in the rest of the paper.

\subsection{Three-dimensional eddy-current problem}
Let $\Omega^{3D} \subset \mathbb{R}^3$ be a bounded domain and $\Sigma^{3D} \subset \Omega^{3D}$ describe the conducting region, which we assume to be connected fr ease of notation.
We consider the $A$-$\phi$ formulation, which is widely used in practice~\cite{egger:Kameari1990}. The governing equations 
are given by
\begin{alignat}{2} 
\sigma (\partial_t \mathbf{A} + \nabla \phi) + \operatorname{curl} ( \partial_B w(\operatorname{curl} \mathbf{A}) )
&= \mathbf{J}_s \qquad && \text{in } \Omega^{3D}\label{egger:eq:3D_1}\\
\operatorname{div} \mathbf{A} &= 0  \qquad && \text{in } \Omega^{3D}.\label{egger:eq:3D_2}
\end{alignat}
Here $\mathbf{A}$ is the magnetic vector potential and $\phi$ the electric scalar potential, 
whose support is restricted to $\Sigma^{3D}$. 
Further $\mathbf{B}=\operatorname{curl}\mathbf{A}$ is the magnetic flux density and $\mathbf{H}=\partial_B w(\mathbf{B})$ the magnetic field intensity. The function $w(\mathbf{B})$ is the magnetic energy density. In isotropic materials, we have $\partial_B w(\mathbf{B}) = \nu(|\mathbf{B}|) \, \mathbf{B}$, with $\nu$ denoting the field-dependent reluctivity. 
The electric conductivity $\sigma$ is non-zero only on $\Sigma^{3D}$, and the source current density $\mathbf{J}_s$ is assumed solenoidal and to vanish in a neighborhood of the conductor $\Sigma^{3D}$. 
The eddy-current density in this model is $\mathbf{J}_{ec}=-\sigma (\partial_t \mathbf{A} + \nabla \phi)$, and the loss density reads
\begin{align}
p^{3D}_{ec} = \tfrac{1}{\sigma} |\mathbf{J}_{ec}|^2 = \sigma |\partial_t \mathbf{A} + \nabla \phi|^2.
\end{align}
Total losses are computed by integration over the conductor $\Sigma^{3D}$. 
From \eqref{egger:eq:3D_1} and appropriate boundary conditions for $\phi$, we further conclude that $\operatorname{div} \mathbf{J}_{ec}=0$ on $\Sigma^{3D}$ and $\mathbf{n} \cdot \mathbf{J}_{ec}=0$ on $\partial\Sigma^{3D}$. These physical conditions are used in the following discussion.
%

\subsection{Two-dimensional eddy-current problem}
If the domain and problem data are independent of the $z$-coordinate, following two-dimensional model is appropriate~\cite{egger:Salon}. We denote by $\Omega \subset \mathbb{R}^2$ the cross-section of the three-dimensional geometry and by $\Sigma \subset \Omega$ the corresponding cross-section of the conductor.
Following \cite{egger:Tsukerman1992,egger:Belahcen2010}, the magnetic fields are described by 
\begin{align}
\sigma \big(\partial_t \mathbf{a} + \overline{ \mathbf{g}}\big) + \mathrm{curl}\big(\partial_{B} w(\mathrm{Curl} \mathbf{a})\big) &= \mathbf{j}_s \quad \text{in } \Omega \label{egger:eq:1}\\
\int_{\Sigma} \sigma \big(\partial_t \mathbf{a} + \overline{\mathbf{g}}\big)\,\mathrm{d}x &= 0. \label{egger:eq:2}
\end{align}
Here $\mathbf{a}$ and $\mathbf{j}_s$ are the axial 
components of the magnetic vector potential and the source current density, and $\overline{\mathbf{g}}$ represents the gradient of the electric scalar potential which is assumed constant on $\Sigma$. 
Further, $\mathbf{b}=\mathrm{Curl}(\mathbf{a})=(\partial_y \mathbf{a}, -\partial_x \mathbf{a})$ represents the in-plane components of the magnetic induction, $\mathbf{h} = \partial_B w(\mathbf{b})$ the corresponding magnetic field components, and $\operatorname{curl} \mathbf{h} = \partial_x \mathbf{h}_y - \partial_y \mathbf{h}_x$ the vector-to-scalar curl in two dimensions. 
The eddy-currents here are $\mathbf{j}_{ec} = -\sigma (\partial_t \mathbf{a} + \overline{\mathbf{g}})$ and the corresponding loss density reads
\begin{align}\label{egger:eq:3}
    p_{ec} = \tfrac{1}{\sigma} |\mathbf{j}_{ec}|^2 = \sigma |\partial_t \mathbf{a} + \overline{\mathbf{g}}|^2.
\end{align}
Together with appropriate initial and boundary conditions, existence of a unique solution to \eqref{egger:eq:1}--\eqref{egger:eq:2} can be established for the transient and time-periodic case; see~\cite{egger:Egger2025}.
For later reference, let us clarify the relations between the two- and the three-dimensional eddy-current problems. 
\begin{theorem}
Let the functions $(\mathbf{a},\overline{\mathbf{g}})$ denote a solution of the two-dimensional eddy-current problem \eqref{egger:eq:1}--\eqref{egger:eq:2}. 
Then $\mathbf{A}=(0,0,\mathbf{a})$ and $\phi = \overline{\mathbf{g}} \, z$ defines a solution of the three-dimensional eddy-current problem \eqref{egger:eq:3D_1}--\eqref{egger:eq:3D_2} with current density $\mathbf{J}_s = (0,0,\mathbf{j}_s)$ 
on the cylindrical domain $\Omega^{3D} = \Omega \times (0,L)$ and $\Sigma^{3D} = \Sigma \times (0,L)$.
\end{theorem}
The claim follows immediately from the construction of the fields $\mathbf{A}$ and $\nabla \phi$. 
In particular, the gauge condition $\operatorname{div} \mathbf{A}=0$ is satisfied automatically.    
\begin{remark}
The construction of the three-dimensional solution implies that the eddy-currents are related by $\mathbf{J}_{ec} = (0,0,\mathbf{j}_{ec})$ and charge conservation $\operatorname{div}(\mathbf{J}_{ec})=0$ holds true. However, the conductor $\Sigma^{3D}$ touches the boundary of $\Omega^{3D}$ in axial direction and therefore $\mathbf{n} \cdot \mathbf{J}_{ec} \ne 0$ on the axial boundaries of $\Sigma^{3D}$. From \eqref{egger:eq:2}, one can only conclude that the total current across these boundaries vanishes, which is the best approximation achievable within the two-dimensional model. In practice, this would mean that the axial ends of the conductor are short-circuited, which is not satisfied in relevant applications and hence generates a consistency error; see \cite{egger:Belahcen2010,egger:Steentjes2015} for further discussion.
\end{remark}

\subsection{Simplified eddy-current model}

While feasible in principle, the computation of eddy-current losses via \eqref{egger:eq:1}--\eqref{egger:eq:2} is still computationally expensive and usually avoided during design optimization. The standard computational practice is to consider instead a sequence of static field problems 
\begin{align}
\operatorname{curl} (\partial_B w(\operatorname{Curl} \hat{\mathbf{a}})) &= \mathbf{j}_s \qquad \text{in } \Omega\label{egger:eq:static}
\end{align}
and to estimate the losses by post-processing. 
%
A first guess would be to define eddy-currents by $\hat{\mathbf{j}}'_{ec}=\sigma \partial_t \hat{\mathbf{a}}$, which would however violate the two-dimensional charge conservation condition $\int_\Sigma \hat{\mathbf{j}}'_{ec} \mathrm{d}x=0$, see \eqref{egger:eq:2}. Following \cite{egger:Deak2008,egger:Belahcen2010}, a better choice is to define $\hat{\mathbf{j}}_{ec}=-\sigma (\partial_t \hat{\mathbf{a}} + \hat{\mathbf{g}})$ with $\hat{\mathbf{g}}$ chosen such that $\int_\Sigma \sigma (\partial_t \hat{\mathbf{a}} + \hat{\mathbf{g}}) \mathrm{d}x = 0$, and to set 
\begin{align} \label{egger:eq:static_ec}
\hat p_{ec} = \tfrac{1}{\sigma} |\hat{\mathbf{j}}_{ec}|^2 = \sigma |\partial_t \hat{\mathbf{a}} + \hat{\mathbf{g}}|^2.    
\end{align}
This approach can be integrated easily as a post-processing step in a standard computational workflow based on two-dimensional magnetostatic simulations.
%

\section{Eddy-current post-processing}

We now present our eddy-current post-processing scheme that takes into account global reaction fields and charge conservation, but avoids the solution of the full nonlinear time-dependent eddy-current problem~\eqref{egger:eq:1}--\eqref{egger:eq:2}.
We start from the magnetostatic field $\hat{\mathbf{a}}$ determined by \eqref{egger:eq:static} and make the ansatz 
\begin{align}  \label{egger:eq:post0}
\tilde{\mathbf{a}} = \hat{\mathbf{a}} + \tilde{\mathbf{z}}
\end{align}
to obtain an improved estimate $\tilde{\mathbf{a}}$ of the two-dimensional vector potential $\mathbf{a}$.
The correction field $\tilde{\mathbf{z}}$ is determined by the linearized eddy-current problem
\begin{align}
\sigma\big(\partial_t \tilde{\mathbf{z}} + \tilde{\mathbf{g}}\big) + \operatorname{curl}\big(
\tilde \nu \operatorname{Curl} \tilde{\mathbf{z}} \big) &= -\sigma \partial_t \hat{\mathbf{a}} \qquad \text{in } \Omega, \label{egger:eq:post1}\\
\int_{\Sigma} \sigma \big(\partial_t \tilde{\mathbf{z}} + \tilde{\mathbf{g}}\big)\,\mathrm{d}x &= - \int_{\Sigma} \sigma \partial_t \hat{\mathbf{a}}\,\mathrm{d}x. \label{egger:eq:post2}
\end{align}
Specific choices for the linearized reluctivity $\tilde{\nu}$ will be discussed below. 
The eddy-currents for this model are given by $\tilde{\mathbf{j}}_{ec} = -\sigma (\partial_t \tilde{\mathbf{a}} +\tilde{\mathbf{g}})$ 
and the resulting losses are
\begin{align} \label{egger:eq:post3}
\tilde{p}_{ec} = \tfrac{1}{\sigma} |\tilde{\mathbf{j}}_{ec}|^2 = \sigma |\partial_t \tilde{\mathbf{a}} + \tilde{\mathbf{g}}|^2.
\end{align}
The model \eqref{egger:eq:post0}--\eqref{egger:eq:post3} involves the solution of a sequence of nonlinear magnetostatic field problems \eqref{egger:eq:static} and the solution of one 
linearized time-periodic eddy-current correction problem \eqref{egger:eq:post1}--\eqref{egger:eq:post2}, which has a similar structure to the nonlinear eddy-current model \eqref{egger:eq:1}--\eqref{egger:eq:2}, but significantly cheaper to solve numerically. 
Together with appropriate boundary and initial conditions, the existence of a unique solution to the linearized problem again follows from the results in~\cite{egger:Egger2025}.
%

\subsection{Theoretical justification} 

We consider a time-periodic setting with homogeneous boundary conditions and assume that unique solvability of the problems \eqref{egger:eq:1}--\eqref{egger:eq:2}, \eqref{egger:eq:static}, and \eqref{egger:eq:post1}--\eqref{egger:eq:post2} is guaranteed.
The corresponding solutions are denoted by $(\mathbf{a},\overline{\mathbf{g}})$, $\hat{\mathbf{a}}$, and $(\tilde{\mathbf{z}},\tilde{\mathbf{g}})$, respectively.
By the superposition principle for linear problems, we immediately obtain the following result.
\begin{theorem} \label{thm:1}
Let $w(\mathbf{b})$ be a quadratic function. 
Then for the choice $\tilde{\nu} = \partial_{BB} w(\operatorname{Curl} \hat{\mathbf{a}})$ for the reluctivity, we have $\mathbf{a}=\tilde{\mathbf{a}}$ and $\bar{\mathbf{g}} = \tilde{\mathbf{g}}$ with $\tilde{\mathbf{a}} = \hat{\mathbf{a}} + \tilde{\mathbf{z}}$ according to \eqref{egger:eq:post0}.
\end{theorem}
The proposed post-processing strategy thus is exact for problems with linear materials. 
The second result clarifies the situation for the nonlinear magnetic regime. 
\begin{theorem} \label{thm:2}
Let $\tilde{\nu} = \partial_{BB} w(\operatorname{Curl} \hat{\mathbf{a}})$ and $\tilde{\mathbf{a}}= \hat{\mathbf{a}} + \tilde{\mathbf{z}}$ as before. Then $(\tilde{\mathbf{a}},\tilde{\mathbf{g}})$ amounts to the first iterate of Newton's method applied to \eqref{egger:eq:1}--\eqref{egger:eq:2} with initial iterate chosen as $(\hat{\mathbf{a}},\mathbf{0})$.
\end{theorem}
\begin{proof}
 To verify the claim, we note that \eqref{egger:eq:1}--\eqref{egger:eq:2} can be phrased as an abstract nonlinear system $F(x)=y$ with $x=(\mathbf{a},\mathbf{g})$ and $y=(\mathbf{j}_s,0)$.
For $x^0=(\hat{\mathbf{a}},\mathbf{0})$, the two components of the residual $r=y-F(x^0)$ are given by
\begin{align*}
 r_1 &= \mathbf{j}_s - \sigma \partial_t \hat{\mathbf{a}} - \operatorname{curl}(\partial_{B} w(\operatorname{Curl} \hat{\mathbf{a}})) = -\sigma \partial_t \hat{\mathbf{a}}\\
 r_2 &= -\int_\Sigma \sigma \partial_t \hat{\mathbf{a}}.
\end{align*}
In the first equation, we used the static problem \eqref{egger:eq:static}.
Newton's method defines $x^1 = x^0 + v$ where $v$ is the solution of the linearized problem $DF(x^0) v = r$, which exactly amounts to the system \eqref{egger:eq:post1}--\eqref{egger:eq:post2} with $\tilde \nu = \partial_{BB} w(\operatorname{Curl} \hat{\mathbf{a}})$ and $v=(\tilde{\mathbf{z}},\tilde{\mathbf{g}})$. 
\end{proof}

\begin{remark}
The proposed correction is exact for linear models and corresponds to the first Newton iterate in the nonlinear case. 
Our numerical results demonstrate that the reaction field can already be recovered by solving a single linearized problem.
%
%
Similar approaches, computing a linearized correction for the reaction field on a three-dimensional sub-domain representing or surrounding the conductor have been considered in \cite{egger:Yamazaki2009,egger:Okitsu2012}. A clear mathematical justification of these approaches would be highly valuable.
\end{remark}

\subsection{Efficient computation of the correction field}

The proposed post-processing strategy requires the solution of the linearized time-periodic eddy-current problem \eqref{egger:eq:post1}--\eqref{egger:eq:post2}. We briefly discuss some standard approaches to solve this problem and note that only a rather moderate accuracy is required to obtain a reasonable estimate of the correction field and the resulting eddy-currents. 

A simple but widely used solution strategy is classical time stepping combined with a cycling procedure to enforce periodicity. In our numerical tests, reasonable approximations of the correction field $(\tilde{\mathbf{z}},\tilde{\mathbf{g}})$ and the resulting eddy-current losses could be obtained after two or three cycles. The resulting algorithm is straightforward to implement, it requires only the repeated solution of generalized magnetostatic field problems, but it is inherently sequential.

Alternatively, one may employ more modern parallel-in-time techniques such as Parareal or multigrid-in-time methods~\cite{egger:Gander2015}. We adopt the latter strategy here which, similar to the time-stepping approach, reduces the solution of the time-periodic problem to the repeated solution of generalized magnetostatic problems. These problems can be treated concurrently, leading to efficient parallel algorithms with excellent convergence and parallel scaling properties~\cite{egger:Egger2025}. 
%

\section{Numerical validation}
\label{egger:sec:4}

We illustrate the proposed approach by computing eddy-current losses for a permanent-magnet synchronous motor (PMSM). A detailed description of the motor geometry and the material properties can be found in \cite{egger:Krenn2025}.
To evaluate the robustness of the methods with respect to problem data, we consider in our simulations a high-frequency, high-current operating point (16000 RPM and current density $23.7$ A/mm\textsuperscript{2}). Qualitatively similar results were obtained for less challenging regimes. 
\begin{figure}[ht!]
    \centering
    \includegraphics[width=.32\linewidth]{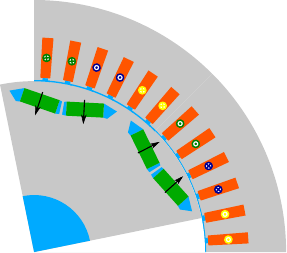}
    \hspace*{2cm}
    \includegraphics[trim={0 0 0 25}, clip, width=.32\linewidth]{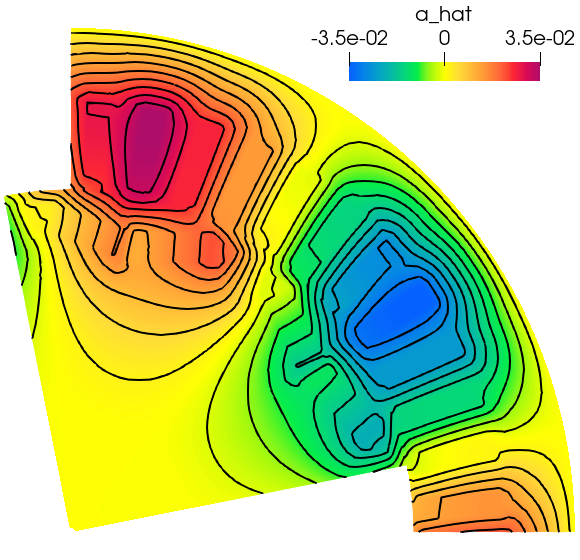}%
    \caption{Benchmark model of a PMSM from \cite{egger:Krenn2025}. Left: Geometric model. Right: Magnetic vector potential $\mathbf{a}$ and corresponding isolines representing magnetic fluxes $\mathbf{b}=\operatorname{Curl}(\mathbf{a})$. }
    \label{fig:pmsm}
\end{figure}
Using symmetry, only one quarter of the motor is simulated; see the left plot of Fig.~\ref{fig:pmsm}. 
The magnetic vector potential is discretized by lowest-order finite elements, combined with implicit Euler time integration using 240 time steps per mechanical quarter revolution. Rotor motion is treated by the harmonic mortar method \cite{egger:Egger2021}. 
The right plot of Fig.~\ref{fig:pmsm} shows one snapshot of a typical solution.

In our numerical tests, we compare four different approaches for eddy-current simulation discussed in the previous sections, namely:
\begin{itemize}\itemindent2em\labelsep2em
\item[(M1)] the full nonlinear time-periodic eddy-current model \eqref{egger:eq:1}--\eqref{egger:eq:2};
\item[(M2)] the magnetostatic post-processing approach \eqref{egger:eq:static}--\eqref{egger:eq:static_ec};
\item[(M3)] a local linear correction restricted to the conducting regions only~\cite{egger:Yamazaki2009,egger:Steentjes2015}; and
\item[(M4)] the proposed global linear correction scheme \eqref{egger:eq:post0}--\eqref{egger:eq:post3}.
\end{itemize}
In Table~\ref{tab:results} we summarize the main results and computation times. 
The computed losses of method (M1) serve as reference values. 
The standard post-processing approach (M2) significantly overestimates the losses, which is to be expected \cite{egger:Belahcen2010}. 
\begin{table}[ht]
\centering
\renewcommand{\arraystretch}{1.25}
\setlength{\tabcolsep}{8pt}
\begin{tabular}{l||r|r|r|r}
Model & M1 & M2 & M3 & M4\\
\hline
\hline
Loss [W]
& 926.44
& 2218.34
& 803.53
& 926.92
\\
\hline 
Time (1 proc.) [s]
& 250.72
& 92.30
& 93.61
& 118.08
\\
\hline 
Time (30 proc.) [s]
& 10.62
& 4.67
& 5.89
& 7.73
\end{tabular}
\caption{Computed losses and execution times for different eddy-current loss models. To evaluate the parallel scaling behaviour, computations times are presented for different number of processors.}
\label{tab:results}
\end{table}
The local correction method (M3) 
slightly underestimates the eddy-current losses since
it neglects the global nature of the reaction field.
The proposed method (M4) reproduces the eddy-current losses with an error below $1\%$, requiring only about $30$--$60\%$ additional computation time compared to the magnetostatic approximation (M2).

\section{Discussion}
\label{egger:sec:5}

As mentioned in the introduction, the proposed splitting strategy (i)--(ii) is generally applicable in two and three space dimensions.
Recovering genuinely three-dimensional effects from two-dimensional simulations, however, remains an open challenge. Existing hybrid approaches \cite{egger:Yamazaki2009,egger:Okitsu2012,egger:Steentjes2015} require substantial additional modelling and computational effort, while their benefits have not yet been demonstrated conclusively. A rigorous justification of such reconstruction strategies by error estimation and asymptotic analysis therefore remains an important topic for future research.
%
\begin{acknowledgement}
This work was supported by the joint DFG/FWF Collaborative Research Centre CREATOR (DFG: Project-ID 492661287/TRR 361; FWF: 10.55776/F90).
\end{acknowledgement}



\begin{thebibliography}{99.}

\bibitem{egger:Kameari1990}
A. Kameari. 
\newblock Calculation of transient 3-D eddy current using edge-
elements. 
\newblock {\em IEEE Trans. Magn.}, 26 (1990), 466--469.

\bibitem{egger:Tsukerman1992}
I. A. Tsukerman, A. Konrad, and J. D. Lavers.
\newblock A method for circuit connections in time-dependent eddy current problems.
\newblock {\em IEEE Trans. Magn.}, 28 (1992),  1299--1302.

\bibitem{egger:Demenko2010}
A. Demenko and K. Hameyer. 
\newblock Field and field-circuit models of electrical machines.
\newblock {\em COMPEL} 29 (2010), 8--22.

\bibitem{egger:Dlala2010}
E. Dlala, A. Belahcen, and A. Arkkio. 
\newblock On the Importance of Incorporating Iron Losses in the Magnetic Field Solution of Electrical Machines.
\newblock {\em IEEE Trans. Magn.} 46 (2010), 3101--3104.

\bibitem{egger:Gyselinck2006} 
J.~Gyselinck, R.~V.~Sabariego, and P.~Dular.
\newblock A nonlinear time-domain homogenization technique for laminated iron cores in three-dimensional FE models. 
\newblock {\em IEEE Trans. Magn.}, 42 (2006), 763--766.

\bibitem{egger:Biro2006}
O. Biro and K. Preis. 
\newblock An efficient time domain method for nonlinear periodic eddy current problems. 
\newblock {\em IEEE Trans. Magn.}, 42 (2006),  695--698.

\bibitem{egger:Takahashi2012}
Y. Takahashi, T. Iwashita, H. Nakashima, T. Tokumasu, M. Fujita, S. Wakao, K. Fujiwara, and Y. Ishihara.
\newblock Parallel time-periodic finite-element method for steady-
state analysis of rotating machines.
\newblock {\em IEEE Trans. Magn.}, 48 (2012),  1019--1022.

\bibitem{egger:Deak2008}
C. Deak, L. Petrovic, A. Binder, M. Mirzaei, D. Irimie, and
B. Funieru. 
\newblock Calculation of eddy current losses in permanent magnets of synchronous machines. 
\newblock In: Proc. SPEEDAM, Jun. 2008. 

\bibitem{egger:Belahcen2010}
A. Belahcen and A. Arkkio. 
\newblock Permanent magnets models and losses in 2D FEM simulation of electrical machines.
In: {\em Proc. of 19th ICEM Conference, Rome}, (2010), pp. 1-6.

\bibitem{egger:Yamazaki2009}
K. Yamazaki and Y. Kanou.
\newblock Rotor loss analysis of interior permanent magnet motors using combination of 2-D and 3-D finite element method.
\newblock {\em IEEE Trans. Magn.}, 45 (2009), 1772--1775.

\bibitem{egger:Okitsu2012}
T. Okitsu, D. Matsuhashi, Y. Gao, and K. Muramatsu.
\newblock Coupled 2-D and 3-D eddy current analyses for evaluating eddy current loss of a permanent magnet in surface PM motors.
\newblock {\em IEEE Trans. Magn.}, 48 (2012), 3100--3103.

\bibitem{egger:Steentjes2015}
S. Steentjes, S. Boehmer, and K. Hameyer.
\newblock Permanent magnet eddy-current losses in 2-D FEM simulations of electrical machines. 
\newblock {\em IEEE Trans. Magn.}, 51 (2015), Art no. 6300404.

\bibitem{egger:Salon}
S. J. Salon.
\newblock {\em Finite Element Analysis of Electrical Machines}.
\newblock Springer, New York, NY, 
1995.

\bibitem{egger:Egger2025}
H. Egger and A. Schafelner.
\newblock A parallel-in-time solver for nonlinear degenerate time-periodic parabolic problems.
\newblock arXiv:2502.21013, to appear in: {\em Proc. of LSSC 2025}.

\bibitem{egger:Gander2015}
M. J. Gander.
\newblock 50 years of time parallel time integration. 
\newblock In: {\em Multiple Shooting and Time
Domain Decomposition Methods}. 
Springer, 2015. 
pp. 69--113. (2015)

\bibitem{egger:Krenn2025}
P. Gangl, H. De Gersem and N. Krenn.
\newblock Multi-material topology optimization of electric machines under maximum temperature and stress constraints.
\newblock {\em Appl. Math. Model.}, 150 (2025), 116347.

\bibitem{egger:Egger2021}
H. Egger, M. Harutyunyan, M. Merkel, and S. Sch\"ops. 
\newblock On the stability of harmonic coupling methods with application to electric machines. 
\newblock In: {\em Proc. of SCEE 2020}, 
Springer, Cham, 2021.   
\end{thebibliography}
\end{document}